\documentclass[11pt]{article}
\usepackage[margin=1.5in]{geometry}

\usepackage{hyperref}
\usepackage{amsmath}
\usepackage{amsthm}
\usepackage{amssymb}
\usepackage{amsfonts}
\usepackage{multicol}
\usepackage{subfigure}
\usepackage[numbers,sort&compress]{natbib}
\usepackage{booktabs}
\usepackage{titlesec}
\usepackage{longtable,tabu}

\usepackage{enumitem}
\setlist[enumerate]{label={\arabic*.}}

\usepackage{color}

\usepackage{tikz}
\usetikzlibrary{calc}

\usepackage{tkz-graph}

\newtheorem{theorem}{Theorem}[section]
\newtheorem{lemma}[theorem]{Lemma}
\newtheorem{proposition}[theorem]{Proposition}

\newtheorem{corollary}[theorem]{Corollary}

\newtheorem{claim}[theorem]{Claim}

\theoremstyle{definition}

\newtheorem{conjecture}[theorem]{Conjecture}
\newtheorem{problem}{Open Problem}

\newcommand{\squishlist}{
 \begin{list}{$\bullet$}
  { \setlength{\itemsep}{0pt}
     \setlength{\parsep}{3pt}
     \setlength{\topsep}{3pt}
     \setlength{\partopsep}{0pt}
     \setlength{\leftmargin}{2.5em}
     \setlength{\labelwidth}{1em}
     \setlength{\labelsep}{0.5em} } }

\newcommand{\squishlisttwo}{
 \begin{list}{$\triangleright$}
  { \setlength{\itemsep}{0pt}
     \setlength{\parsep}{0pt}
    \setlength{\topsep}{0pt}
    \setlength{\partopsep}{0pt}
    \setlength{\leftmargin}{2em}
    \setlength{\labelwidth}{1.5em}
    \setlength{\labelsep}{0.5em} } }

\newcommand{\squishend}{
  \end{list}  }
  
\newcommand{\kcol}{$k$-\textsc{Colouring}}

\begin{document}

\title{Vertex-critical graphs in co-gem-free graphs}
\author{
 Iain Beaton\\ 
 \small Department of Mathematics and Statistics\\
 \small Acadia University\\
 \small Wolfville, NS Canada\\
 \small iain.beaton@acadiau.ca\\
 \and
 Ben Cameron\\ 
 \small School of Mathematical and Computational Sciences\\
 \small University of Prince Edward Island\\
 \small Charlottetown, PE Canada\\
 \small brcameron@upei.ca\\
}

\date{\today}

\maketitle

\begin{abstract}
A graph $G$ is $k$-vertex-critical if $\chi(G)=k$ but $\chi(G-v)<k$ for all $v\in V(G)$ and $(G,H)$-free if it contains no induced subgraph isomorphic to $G$ or $H$. We show that there are only finitely many $k$-vertex-critical (co-gem, $H$)-free graphs for all $k$ when $H$ is any graph of order $4$ by showing finiteness in the three remaining open cases, those are the cases when $H$ is $2P_2$, $K_3+P_1$, and $K_4$. 
For the first two cases we actually prove the stronger results:
\begin{itemize}
    \item There are only finitely many $k$-vertex-critical (co-gem, paw$+P_1$)-free graphs for all $k$ and that only finitely many $k$-vertex-critical (co-gem, paw$+P_1$)-free graphs for all $k\ge 1$.
    \item  There are only finitely many $k$-vertex-critical (co-gem, $P_5$, $P_3+cP_2$)-free graphs for all $k\ge 1$ and $c\ge 0$.
\end{itemize}
To prove the latter result, we employ a novel application of Sperner's Theorem on the number of antichains in a partially ordered set. Our result for $K_4$ uses exhaustive computer search and is proved by showing the stronger result that every $(\text{co-gem, }K_4)$-free graph is $4$-colourable. Our results imply the existence of simple polynomial-time certifying algorithms to decide the $k$-colourability of (co-gem, $H$)-free graphs for all $k$ and all $H$ of order $4$ by searching the vertex-critical graphs as induced subgraphs. 

\end{abstract}

\section{Introduction}
For a fixed $k$, we say a graph $G$ is $k$-colourable if there exists a function $\phi:V(G) \to \{1,2,\ldots k\}$ such that $\phi(u) \neq \phi(v)$ for every edge $uv \in E(G)$. The \kcol{} decision problem takes a graph $G$ as input and returns true or false based on the $k$-colourability of $G$. In 1972, Karp~\cite{Karp1972} showed for every fixed $k\ge 3$ that \kcol{} is NP-complete. When the input graphs are restricted to come from a specific family, however, polynomial-time algorithms can sometimes be developed. One such type of family is that of $H$-free graphs which are those that do not contain the graph $H$ as an induced subgraph. Ho\`{a}ng et al.~\cite{Hoang2010} gave polynomial-time algorithms to solve \kcol{} for all $k$ for the family of $P_5$-free graphs. In fact, $P_5$ is the largest connected graph that can be forbidden as an induced subgraph where \kcol{} can be solved in polynomial-time for all $k$ (assuming P$\neq$NP). This is because \kcol{} remains NP-complete for claw-free graphs~\cite{Holyer1981,LevenGail1983},  $C_n$-free graphs for any $n\ge 3$~\cite{KaminskiLozin2007} (for $k\ge 3$), and $P_6$-free graphs~\cite{Huang2016} (for $k\ge 5$, but polynomial-time solvable for $P_6$-free graphs when $k \le 4$~\cite{P6freeconf,P6free1,P6free2}). Beyond forbidding connected induced graphs, it is known that \kcol{} can be solved in polynomial-time for all $k\ge 1$ and $r\ge 0$ for $H$-free graphs if $H$ is $P_5+rP_1$~\cite{Couturier2015} or $rP_3$~\cite{ChudnovskyHajebiSpirkl2024}. On the other hand, \kcol{} remains NP-complete for all $k\ge 5$ for $H$-free graphs when $H$ is $P_5+P_2$~\cite{ChudnovskyHuangSpirklZhong2021} or $2P_4$~\cite{HajebiLiSpirkl2022}. Thus, for all $k\ge 5$, it remains only to determine the complexity class of $(P_4+rP_3)$-free graphs for all $r \ge 1$.

In some, more restricted families than those mentioned above, polynomial-time \textit{certifying} algorithms to solve \kcol{} for any graph in the family can be obtained. An algorithm is said to be certifying if it returns an easily verifiable certificate of correctness with each output (see~\cite{McConnell2011}, for a survey).  For \kcol{} a certificate for positive output is a $k$-colouring, and for negative output, it is a $(k+1)$-vertex-critical induced subgraph where a graph is $k$-vertex-critical. A graph is $k$-vertex-critical if every proper induced subgraph of it is $(k-1)$-colourable, but it requires $k$ colours. Some of the algorithms listed above do certify positive output, for example the algorithms for $P_5$-free graphs~\cite{Hoang2010} and $(P_5+rP_1)$-free graphs~\cite{Couturier2015}, but none of them certify negative output. If a family of graphs contains only finitely many $(k+1)$-vertex-critical graphs, then a polynomial-time algorithm to solve \kcol{} that certifies negative output can be readily implemented by searching the input graph for each of the $(k+1)$-vertex-critical graphs as induced subgraphs, returning such an induced subgraph if one is found (see~\cite{P5banner2019}, for example). Thus, as expected, there are infinitely many $k$-vertex-critical claw-free graphs for all $k\ge 3$ (see \cite{Chud4critical2020} for a discussion of how such an infinite family follows from \cite{LazebnikUstimenko1995}, or the simple construction in~\cite{CameronHoangSawada2022}) and for $H$-free graphs when $H$ contains any cycle (this follows from a classical result of Erd\H{o}s~\cite{Erdos} that for each $k$, there exist $k$-chromatic graphs with arbitrarily large girth). Thus, the only $H$-free families of graphs which may contain only finitely many $k$-vertex-critical graphs for any $k\ge 4$ are those where every component of $H$ is a path. Since $P_4$-free graphs are perfect, the only such $k$-vertex-critical graph is $K_k$ for all $k$. For $P_5$-free graphs, there are only finitely many $4$-vertex-critical graphs~\cite{Hoang2015} and this finite list was used to develop a linear-time algorithm to solve $3$-\textsc{Colouring} for $P_5$-free graphs~\cite{MaffrayMorel2012}. However, for all $k\ge 5$, there are infinitely many $k$-vertex-critical $P_5$-free graphs~\cite{Hoang2015}. A construction for infinitely many $4$-vertex-critical $P_7$-free graphs was given in~\cite{Chudnovsky4criticalconnected2020}. For disconnected graphs where each component is a path, there are infinitely many $k$-vertex-critical $2P_2$-free graph for all $k\ge 5$~\cite{Hoang2015}, but there are only finitely many $k$-vertex-critical $(P_3+\ell P_1)$-free graphs for all $k\ge 1$ and $\ell \ge 0$~\cite{AbuadasCameronHoangSawada2022}. The most comprehensive result to this end is the following theorem.

\begin{theorem}
    [\cite{Chud4critical2020}]
    Let $H$ be a graph. There are only finitely many $4$-vertex-critical $H$-free graphs if and only if $H$ is an induced subgraph of $P_6$, $2P_3$, or $P_4+\ell P_1$ for some natural number $\ell$.
\end{theorem}

From this result and the results on $2P_2$-free and $(P_3+\ell P_1)$-free graphs cited above, the only graphs $H$ where the finiteness of $k$-vertex-critical $H$-free graphs remains unknown for any $k\ge 5$ is $H=P_4+\ell P_1$ for all $\ell\ge 1$.  In this paper, we will look at subfamilies of the smallest remaining open case, $(P_4+P_1)$-free graphs, which we will refer to by their more commonly used name, co-gem-free graphs from now on (the gem being the complement of $P_4+P_1$). For this family of graphs, the only result that is explicitly known for co-gem-free graphs is that there are only finitely many $k$-vertex-critical (gem, co-gem)-free graphs for all $k\ge 1$~\cite{AbuadasCameronHoangSawada2022}, but there are some other results that follow as corollaries of results on vertex-critical $P_6$-free graphs since co-gem is an induced subgraph of $P_6$. Our work follows the lead of the extensive study of $P_5$-free graphs, where considering subfamilies of the form $(P_5,H)$-free graphs has been a topic of much research. One of the highlights of this area of research is the dichotomy theorem from~\cite{KCameron2021} that there are only finitely many $k$-vertex-critical $(P_5,H)$-free graphs for all $k\ge 1$ and for $H$ of order four if and only if $H$ is not $2P_2$ or $K_3+P_1$. Analogous to this, our main result in this paper is that there are only finitely many $k$-vertex-critical (co-gem, $H$)-free graphs for all $k\ge 1$ and all graphs $H$ of order four (See Figure~\ref{fig:graphsoforder4} for all nonisomorphic graphs of order four). 

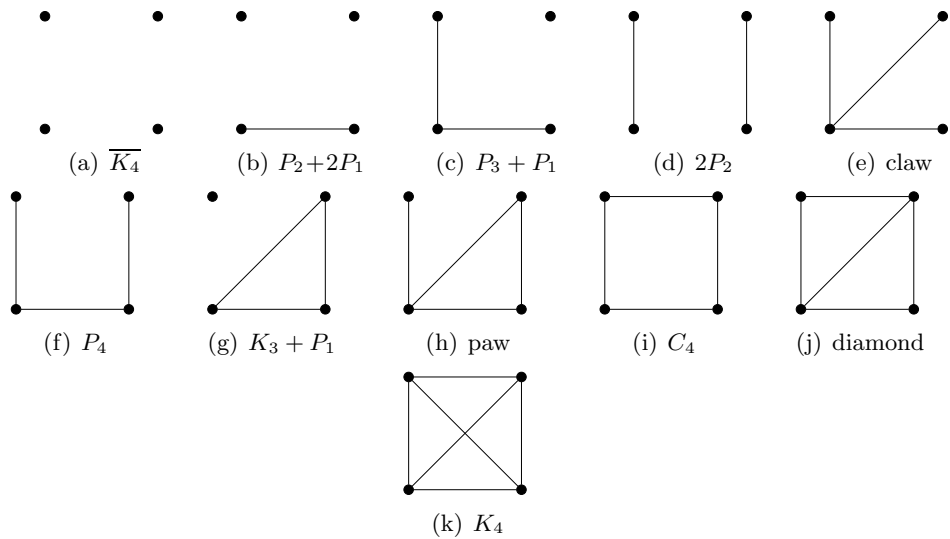
\begin{figure}[h]
    \centering
    \def\c{0.3}
\qquad\subfigure[$\overline{K_4}$]{\scalebox{\c}{\begin{tikzpicture}
\GraphInit[vstyle=Classic]
\Vertex[L=\hbox{},x=5.0cm,y=0.0cm]{v0}
\Vertex[L=\hbox{},x=0.0cm,y=5.0cm]{v1}
\Vertex[L=\hbox{},x=5.0cm,y=5.0cm]{v2}
\Vertex[L=\hbox{},x=0.0cm,y=0.0cm]{v3}
\end{tikzpicture}}}
\qquad\subfigure[$P_2+2P_1$]{\scalebox{\c}{\begin{tikzpicture}
\GraphInit[vstyle=Classic]
\Vertex[L=\hbox{},x=5.0cm,y=0.0cm]{v0}
\Vertex[L=\hbox{},x=0.0cm,y=5.0cm]{v1}
\Vertex[L=\hbox{},x=5.0cm,y=5.0cm]{v2}
\Vertex[L=\hbox{},x=0.0cm,y=0.0cm]{v3}
\Edge[](v0)(v3)
\end{tikzpicture}}}
\qquad\subfigure[$P_3+P_1$]{\scalebox{\c}{\begin{tikzpicture}
\GraphInit[vstyle=Classic]
\Vertex[L=\hbox{},x=5.0cm,y=0.0cm]{v0}
\Vertex[L=\hbox{},x=0.0cm,y=5.0cm]{v1}
\Vertex[L=\hbox{},x=5.0cm,y=5.0cm]{v2}
\Vertex[L=\hbox{},x=0.0cm,y=0.0cm]{v3}
\Edge[](v0)(v3)
\Edge[](v1)(v3)
\end{tikzpicture}}}
\qquad\subfigure[$2P_2$]{\scalebox{\c}{\begin{tikzpicture}
\GraphInit[vstyle=Classic]
\Vertex[L=\hbox{},x=5.0cm,y=0.0cm]{v0}
\Vertex[L=\hbox{},x=0.0cm,y=5.0cm]{v1}
\Vertex[L=\hbox{},x=5.0cm,y=5.0cm]{v2}
\Vertex[L=\hbox{},x=0.0cm,y=0.0cm]{v3}
\Edge[](v0)(v2)
\Edge[](v1)(v3)
\end{tikzpicture}}}
\qquad\subfigure[claw]{\scalebox{\c}{\begin{tikzpicture}
\GraphInit[vstyle=Classic]
\Vertex[L=\hbox{},x=5.0cm,y=0.0cm]{v0}
\Vertex[L=\hbox{},x=0.0cm,y=5.0cm]{v1}
\Vertex[L=\hbox{},x=5.0cm,y=5.0cm]{v2}
\Vertex[L=\hbox{},x=0.0cm,y=0.0cm]{v3}
\Edge[](v0)(v3)
\Edge[](v1)(v3)
\Edge[](v2)(v3)
\end{tikzpicture}}}
\qquad\subfigure[$P_4$]{\scalebox{\c}{\begin{tikzpicture}
\GraphInit[vstyle=Classic]
\Vertex[L=\hbox{},x=5.0cm,y=0.0cm]{v0}
\Vertex[L=\hbox{},x=0.0cm,y=5.0cm]{v1}
\Vertex[L=\hbox{},x=5.0cm,y=5.0cm]{v2}
\Vertex[L=\hbox{},x=0.0cm,y=0.0cm]{v3}
\Edge[](v0)(v2)
\Edge[](v0)(v3)
\Edge[](v1)(v3)
\end{tikzpicture}}}
\qquad\subfigure[$K_3+P_1$]{\scalebox{\c}{\begin{tikzpicture}
\GraphInit[vstyle=Classic]
\Vertex[L=\hbox{},x=5.0cm,y=0.0cm]{v0}
\Vertex[L=\hbox{},x=0.0cm,y=5.0cm]{v1}
\Vertex[L=\hbox{},x=5.0cm,y=5.0cm]{v2}
\Vertex[L=\hbox{},x=0.0cm,y=0.0cm]{v3}
\Edge[](v0)(v2)
\Edge[](v0)(v3)
\Edge[](v2)(v3)
\end{tikzpicture}}}
\qquad\subfigure[paw]{\scalebox{\c}{\begin{tikzpicture}
\GraphInit[vstyle=Classic]
\Vertex[L=\hbox{},x=5.0cm,y=0.0cm]{v0}
\Vertex[L=\hbox{},x=0.0cm,y=5.0cm]{v1}
\Vertex[L=\hbox{},x=5.0cm,y=5.0cm]{v2}
\Vertex[L=\hbox{},x=0.0cm,y=0.0cm]{v3}
\Edge[](v0)(v2)
\Edge[](v0)(v3)
\Edge[](v1)(v3)
\Edge[](v2)(v3)
\end{tikzpicture}}}
\qquad\subfigure[$C_4$]{\scalebox{\c}{\begin{tikzpicture}
\GraphInit[vstyle=Classic]
\Vertex[L=\hbox{},x=5.0cm,y=0.0cm]{v0}
\Vertex[L=\hbox{},x=0.0cm,y=5.0cm]{v1}
\Vertex[L=\hbox{},x=5.0cm,y=5.0cm]{v2}
\Vertex[L=\hbox{},x=0.0cm,y=0.0cm]{v3}
\Edge[](v0)(v2)
\Edge[](v0)(v3)
\Edge[](v1)(v2)
\Edge[](v1)(v3)
\end{tikzpicture}}}
\qquad\subfigure[diamond]{\scalebox{\c}{\begin{tikzpicture}
\GraphInit[vstyle=Classic]
\Vertex[L=\hbox{},x=5.0cm,y=0.0cm]{v0}
\Vertex[L=\hbox{},x=0.0cm,y=5.0cm]{v1}
\Vertex[L=\hbox{},x=5.0cm,y=5.0cm]{v2}
\Vertex[L=\hbox{},x=0.0cm,y=0.0cm]{v3}
\Edge[](v0)(v2)
\Edge[](v0)(v3)
\Edge[](v1)(v2)
\Edge[](v1)(v3)
\Edge[](v2)(v3)
\end{tikzpicture}}}
\qquad\subfigure[$K_4$]{\scalebox{\c}{\begin{tikzpicture}
\GraphInit[vstyle=Classic]
\Vertex[L=\hbox{},x=5.0cm,y=0.0cm]{v0}
\Vertex[L=\hbox{},x=0.0cm,y=5.0cm]{v1}
\Vertex[L=\hbox{},x=5.0cm,y=5.0cm]{v2}
\Vertex[L=\hbox{},x=0.0cm,y=0.0cm]{v3}
\Edge[](v0)(v1)
\Edge[](v0)(v2)
\Edge[](v0)(v3)
\Edge[](v1)(v2)
\Edge[](v1)(v3)
\Edge[](v2)(v3)
\end{tikzpicture}}}

    \caption{All 11 nonisomorphic graphs of order four.}
    \label{fig:graphsoforder4}
\end{figure}

\noindent For two of these graphs $H$ of order four we actually prove stronger results for some of their supergraphs of order at least five. Considering vertex-critical (co-gem, $H$)-free graphs for graphs $H$ of order five has a large body of analogous research for $(P_5, H)$-free graphs. In~\cite{KCameron2021}, the open problem was posed to extend the order-four dichotomy to one of order five and finiteness has been shown for all $k\ge 1$ for many such graphs $H$ including 
banner~\cite{Brause2022}, $K_{2,3}$ and $K_{1,4}$~\cite{Kaminski2019}, $P_2 + 3P_1$~\cite{CameronHoangSawada2022}, $P_3 + 2P_1$~\cite{AbuadasCameronHoangSawada2022}, $\overline{P_5}$~\cite{Dhaliwal2017}, $\overline{P_3 + P_2}$ and gem~\cite{CaiGoedgebeurHuang2021} (see also \cite{CameronHoang2023}), dart~\cite{Xiaetal2023}, $K_{1,3} + P_1$ and $\overline{K_3 + 2P_1}$~\cite{xia2024results}. It was also shown that there are infinitely many $k$-vertex-critical $(P_5,C_5)$-free graphs for all $k\ge 6$~\cite{CameronHoang2023}. 

\subsection{Outline}
The paper proceeds as follows. We prove that there are only finitely many $k$-vertex-critical (co-gem, $P_5, P_3+cP_2$)-free graphs for all $k\ge 1$ and $c\ge 1$ in Section~\ref{sec:2P2}, from which it follows as a corollary that there are only finitely many $k$-vertex-critical (co-gem, $2P_2$)-free graphs. Our proofs in this section use the result on vertex-critical ($P_3+\ell P_1$)-free graphs from~\cite{AbuadasCameronHoangSawada2022} as well as a novel application of Sperner's Theorem that we expect will be of interest. In Section~\ref{sec:traingleplusP1}, we prove that there are only finitely many $k$-vertex-critical (co-gem, paw$+P_1$)-free graphs for all $k\ge 1$, from which it follows that only finitely many $k$-vertex-critical (co-gem, $K_3+P_1$)-free graphs for all $k\ge 1$. Using the exhaustive computer search techniques developed in~\cite{Hoang2015} and refined and optimized in \cite{GoedgebeurSchaudt2018}, we show that there are no $5$-vertex-critical (co-gem, $K_4$)-free graphs, and therefore that all such graphs are $4$-colourable.  All of our results, together with previous work, are compiled together in Section~\ref{sec:proofofmainthm} to prove our main theorem:

\begin{theorem}\label{thm:cogemHfreeord4}
    There are only finitely many $k$-vertex-critical $(\text{co-gem, }H)$-free graphs for all $k\ge 1$ and all graphs $H$ of order four.
\end{theorem}

We then conclude by posing some open questions and giving some data on $k$-vertex-critical graphs in co-gem-free graphs for small $k$ in Section~\ref{sec:conclusion}. We first conclude this section with a brief subsection on notations and definitions and another subsection on preliminary results that will be applied throughout the paper.

\subsection{Notation and Definitions}

If vertices $u$ and $v$ are adjacent in a graph we write $u\sim v$ and if they are nonadjacent we write $u\nsim v$.
For a vertex $v$ in a graph, $N(v)$ and $N[v]$ denote the open neighbourhood and closed neighbourhood of $v$, respectively. More precisely, $N(v)=\{u\in V(G): u\sim v\}$ and $N[v]=N(v)\cup\{v\}$.
For subsets $A$ and $B$ of $V(G)$, we say $A$ is \textit{(anti)complete} to $B$ if $a$ is (non)adjacent to $b$ for all $a\in A$ and $b\in B$.
If $A=\{a\}$ then we simplify notation and say $a$ is (anti)complete to $B$.
For a vertex $v\in V(G)$ and a set $A\subseteq V(G)$, we say that $v$ is \textit{mixed} on $A$ if $v$ is neither complete nor anticomplete to $A$.
We use $\chi(G)$ to denote the \textit{chromatic number} of $G$.
For graphs $G$ and $H$, we let $G+H$ denote their disjoint union, and $\ell G$ denote disjoint union of $G$ with itself $\ell$ times. For a graph $G$ and $S\subseteq V(G)$, we use $G[S]$ to denote the subgraph of $G$ induced by $S$.

\section{Preliminaries}\label{sec:prelims}

We will make extensive use of the following lemma and theorem throughout the paper.

\begin{lemma}[\cite{Hoang2015}]\label{lem:nocomparablecliques}
Let $G$ be a graph with chromatic number $k$. If G contains two disjoint $m$-cliques $A = \{a_1, a_2,\ldots , a_m\}$ and $B = \{b_1, b_2,\ldots , b_m\}$ such that $N(a_i) \setminus A \subseteq N(b_i) \setminus B$ for all $1 \le i \le m$, then $G$ is not $k$-vertex-critical.
\end{lemma}

\noindent We stated Lemma~\ref{lem:nocomparablecliques} here in its full generality for interested readers, but we will only use it for the case where $m=1$. For easier reference in this case, we call vertices $a$ and $b$ \textit{comparable} if $N(a)\subseteq N(b)$. The contrapositive of Lemma~\ref{lem:nocomparablecliques} for $m=1$, can then be restated as there are no comparable vertices in a vertex-critical graph.

A recent theorem due to Abuadas et al.~\cite{AbuadasCameronHoangSawada2022} will be used to reduce the problem of bounding the order of vertex-critical graphs in a family to showing that they are $(P_3+cP_1)$-free for some constant $c$.

\begin{theorem}[\cite{AbuadasCameronHoangSawada2022}]\label{thm:finiteP3ellP1freecrit}
There are only finitely many $k$-vertex-critical $(P_3+\ell P_1)$-free graphs for all $k\ge 1$ and $\ell \ge 0$.
\end{theorem}

One of our results will require showing a given set is sufficiently large, for which we will require Sperner's Theorem. To state this we need some definitions on partially ordered sets. Let $\mathcal{B}\subseteq \mathcal{P}(A)$ be a collection of subsets of a set $A$. Then the relation $\subseteq$ on elements of $\mathcal{B}$ is a partial order. A collection of sets within $\mathcal{B}$ such that no set is a subset of another (i.e. not comparable) is called an \textit{antichain}.

\begin{theorem}    [Sperner's Theorem~\cite{Sperner}]\label{thm:Sperner}

    If $A$ is a set of cardinality $n$ and $\mathcal{B}\subseteq \mathcal{P}(A)$, then the largest antichain in $\mathcal{B}$ has cardinality at most $$\binom{n}{\lfloor\frac{n}{2}\rfloor}.$$
\end{theorem}

As mentioned above, the following result is known on vertex-critical co-gem-free graphs and will be used in the proof of Theorem~\ref{thm:cogemHfreeord4}.
\begin{theorem}[\cite{AbuadasCameronHoangSawada2022}]\label{thm:gemcogem}
    There are only finitely many $k$-vertex-critical (gem, co-gem)-free graphs for all $k\ge 1$.
\end{theorem}

Finally, we need the following result to handle the $4$-vertex-critical case.

\begin{theorem}[\cite{Chudnovsky4criticalconnected2020}]\label{thm:4critP6free}
    There are exactly 80 $4$-vertex-critical $P_6$-free graphs.
\end{theorem}

The list of all 80 graphs was also made available in~\cite{Chudnovsky4criticalconnected2020} and those lists can be searched for all that are further co-gem-free leading to the following corollary

\begin{corollary}\label{cor:finitelymany4critcogem}
    There are exactly nine $4$-vertex-critical co-gem-free graphs (and they are shown in Figure~\ref{fig:4critcogemfree}).
\end{corollary}

\begin{figure}[h]
    \centering
    \def\c{0.25}
    \scalebox{\c}{\qquad \begin{tikzpicture}
\GraphInit[vstyle=Classic]
\Vertex[L=\hbox{},x=4.4689cm,y=5.0cm]{v0}
\Vertex[L=\hbox{},x=5.0cm,y=0.2522cm]{v1}
\Vertex[L=\hbox{},x=0.0cm,y=4.6056cm]{v2}
\Vertex[L=\hbox{},x=0.3001cm,y=0.0cm]{v3}
\Edge[](v0)(v1)
\Edge[](v0)(v2)
\Edge[](v0)(v3)
\Edge[](v1)(v2)
\Edge[](v1)(v3)
\Edge[](v2)(v3)
\end{tikzpicture} }
\scalebox{\c}{\qquad \begin{tikzpicture}
\GraphInit[vstyle=Classic]
\Vertex[L=\hbox{},x=0.3426cm,y=0.8411cm]{v0}
\Vertex[L=\hbox{},x=3.5109cm,y=0.0cm]{v1}
\Vertex[L=\hbox{},x=5.0cm,y=2.7382cm]{v2}
\Vertex[L=\hbox{},x=2.9127cm,y=5.0cm]{v3}
\Vertex[L=\hbox{},x=0.0cm,y=3.8199cm]{v4}
\Vertex[L=\hbox{},x=2.3055cm,y=2.5075cm]{v5}
\Edge[](v0)(v1)
\Edge[](v0)(v4)
\Edge[](v0)(v5)
\Edge[](v1)(v2)
\Edge[](v1)(v5)
\Edge[](v2)(v3)
\Edge[](v2)(v5)
\Edge[](v3)(v4)
\Edge[](v3)(v5)
\Edge[](v4)(v5)
\end{tikzpicture} }
\scalebox{\c}{\qquad \begin{tikzpicture}
\GraphInit[vstyle=Classic]
\Vertex[L=\hbox{},x=0.0cm,y=2.5305cm]{v0}
\Vertex[L=\hbox{},x=2.1757cm,y=1.5051cm]{v1}
\Vertex[L=\hbox{},x=5.0cm,y=1.1986cm]{v2}
\Vertex[L=\hbox{},x=4.2871cm,y=3.3877cm]{v3}
\Vertex[L=\hbox{},x=1.8277cm,y=5.0cm]{v4}
\Vertex[L=\hbox{},x=0.3891cm,y=4.2472cm]{v5}
\Vertex[L=\hbox{},x=3.9951cm,y=0.0cm]{v6}
\Edge[](v0)(v1)
\Edge[](v0)(v4)
\Edge[](v0)(v5)
\Edge[](v1)(v2)
\Edge[](v1)(v5)
\Edge[](v1)(v6)
\Edge[](v2)(v3)
\Edge[](v2)(v6)
\Edge[](v3)(v4)
\Edge[](v3)(v6)
\Edge[](v4)(v5)
\end{tikzpicture} }
\scalebox{\c}{\qquad \begin{tikzpicture}
\GraphInit[vstyle=Classic]
\Vertex[L=\hbox{},x=5.0cm,y=0.3745cm]{v0}
\Vertex[L=\hbox{},x=4.9738cm,y=2.3164cm]{v1}
\Vertex[L=\hbox{},x=3.6224cm,y=5.0cm]{v2}
\Vertex[L=\hbox{},x=0.0cm,y=4.3711cm]{v3}
\Vertex[L=\hbox{},x=1.1996cm,y=1.8168cm]{v4}
\Vertex[L=\hbox{},x=2.7396cm,y=0.0cm]{v5}
\Vertex[L=\hbox{},x=2.4821cm,y=3.6305cm]{v6}
\Edge[](v0)(v1)
\Edge[](v0)(v4)
\Edge[](v0)(v5)
\Edge[](v1)(v2)
\Edge[](v1)(v5)
\Edge[](v1)(v6)
\Edge[](v2)(v3)
\Edge[](v2)(v6)
\Edge[](v3)(v4)
\Edge[](v3)(v6)
\Edge[](v4)(v5)
\Edge[](v4)(v6)
\end{tikzpicture} }
\scalebox{\c}{\qquad \begin{tikzpicture}
\GraphInit[vstyle=Classic]
\Vertex[L=\hbox{},x=0.0cm,y=1.4171cm]{v0}
\Vertex[L=\hbox{},x=3.1171cm,y=0.0cm]{v1}
\Vertex[L=\hbox{},x=5.0cm,y=2.8386cm]{v2}
\Vertex[L=\hbox{},x=3.7838cm,y=5.0cm]{v3}
\Vertex[L=\hbox{},x=1.7073cm,y=3.0913cm]{v4}
\Vertex[L=\hbox{},x=4.4697cm,y=1.1885cm]{v5}
\Vertex[L=\hbox{},x=0.2295cm,y=3.9591cm]{v6}
\Edge[](v0)(v1)
\Edge[](v0)(v4)
\Edge[](v0)(v5)
\Edge[](v0)(v6)
\Edge[](v1)(v2)
\Edge[](v1)(v5)
\Edge[](v2)(v3)
\Edge[](v2)(v5)
\Edge[](v2)(v6)
\Edge[](v3)(v4)
\Edge[](v3)(v6)
\Edge[](v4)(v5)
\Edge[](v4)(v6)
\end{tikzpicture} }
\scalebox{\c}{\qquad \begin{tikzpicture}
\GraphInit[vstyle=Classic]
\Vertex[L=\hbox{},x=0.8263cm,y=0.0cm]{v0}
\Vertex[L=\hbox{},x=2.4453cm,y=0.5481cm]{v1}
\Vertex[L=\hbox{},x=5.0cm,y=2.9458cm]{v2}
\Vertex[L=\hbox{},x=2.5793cm,y=5.0cm]{v3}
\Vertex[L=\hbox{},x=2.0284cm,y=2.5589cm]{v4}
\Vertex[L=\hbox{},x=4.0308cm,y=0.3391cm]{v5}
\Vertex[L=\hbox{},x=0.0cm,y=2.7471cm]{v6}
\Edge[](v0)(v1)
\Edge[](v0)(v4)
\Edge[](v0)(v5)
\Edge[](v0)(v6)
\Edge[](v1)(v2)
\Edge[](v1)(v5)
\Edge[](v1)(v6)
\Edge[](v2)(v3)
\Edge[](v2)(v5)
\Edge[](v3)(v4)
\Edge[](v3)(v6)
\Edge[](v4)(v5)
\end{tikzpicture} }
\scalebox{\c}{\qquad \begin{tikzpicture}
\GraphInit[vstyle=Classic]
\Vertex[L=\hbox{},x=4.4933cm,y=1.8725cm]{v0}
\Vertex[L=\hbox{},x=1.0838cm,y=0.7615cm]{v1}
\Vertex[L=\hbox{},x=0.0cm,y=2.1893cm]{v2}
\Vertex[L=\hbox{},x=3.1097cm,y=3.0424cm]{v3}
\Vertex[L=\hbox{},x=5.0cm,y=5.0cm]{v4}
\Vertex[L=\hbox{},x=1.7413cm,y=4.0928cm]{v5}
\Vertex[L=\hbox{},x=2.7911cm,y=0.0cm]{v6}
\Edge[](v0)(v1)
\Edge[](v0)(v4)
\Edge[](v0)(v5)
\Edge[](v0)(v6)
\Edge[](v1)(v2)
\Edge[](v1)(v5)
\Edge[](v1)(v6)
\Edge[](v2)(v3)
\Edge[](v2)(v5)
\Edge[](v2)(v6)
\Edge[](v3)(v4)
\Edge[](v3)(v6)
\Edge[](v4)(v5)
\end{tikzpicture} }
\scalebox{\c}{\qquad \begin{tikzpicture}
\GraphInit[vstyle=Classic]
\Vertex[L=\hbox{},x=2.8205cm,y=2.8644cm]{v0}
\Vertex[L=\hbox{},x=0.7925cm,y=3.2573cm]{v1}
\Vertex[L=\hbox{},x=0.0cm,y=0.5995cm]{v2}
\Vertex[L=\hbox{},x=3.3706cm,y=0.0cm]{v3}
\Vertex[L=\hbox{},x=5.0cm,y=2.7317cm]{v4}
\Vertex[L=\hbox{},x=3.278cm,y=5.0cm]{v5}
\Vertex[L=\hbox{},x=1.5061cm,y=0.7719cm]{v6}
\Edge[](v0)(v1)
\Edge[](v0)(v4)
\Edge[](v0)(v5)
\Edge[](v0)(v6)
\Edge[](v1)(v2)
\Edge[](v1)(v5)
\Edge[](v1)(v6)
\Edge[](v2)(v3)
\Edge[](v2)(v6)
\Edge[](v3)(v4)
\Edge[](v3)(v6)
\Edge[](v4)(v5)
\end{tikzpicture} }
\scalebox{\c}{\qquad \begin{tikzpicture}
\GraphInit[vstyle=Classic]
\Vertex[L=\hbox{},x=2.5528cm,y=5.0cm]{v0}
\Vertex[L=\hbox{},x=3.9705cm,y=0.0cm]{v1}
\Vertex[L=\hbox{},x=0.0cm,y=3.8907cm]{v2}
\Vertex[L=\hbox{},x=5.0cm,y=1.9255cm]{v3}
\Vertex[L=\hbox{},x=0.6384cm,y=2.0489cm]{v4}
\Vertex[L=\hbox{},x=4.2868cm,y=3.7356cm]{v5}
\Vertex[L=\hbox{},x=-0.2868cm,y=0.0823cm]{v6}
\Edge[](v0)(v2)
\Edge[](v0)(v3)
\Edge[](v0)(v4)
\Edge[](v0)(v5)
\Edge[](v1)(v3)
\Edge[](v1)(v4)
\Edge[](v1)(v5)
\Edge[](v1)(v6)
\Edge[](v2)(v4)
\Edge[](v2)(v5)
\Edge[](v2)(v6)
\Edge[](v3)(v5)
\Edge[](v3)(v6)
\Edge[](v4)(v6)
\end{tikzpicture} }
    \caption{All $4$-vertex-critical co-gem-free graphs.}
    \label{fig:4critcogemfree}
\end{figure}
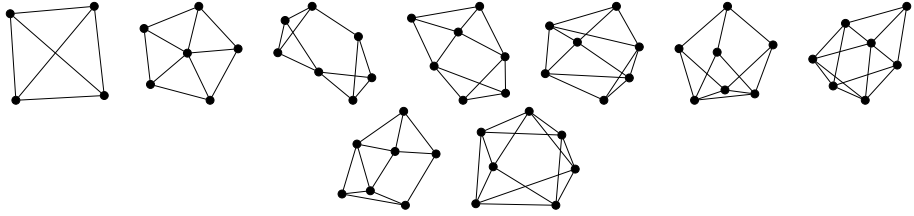

\section{$(\text{co-gem},P_5, P_3+cP_2)$-free graphs}\label{sec:2P2}

We note that the following theorem makes progress on the open problem about the finiteness of $k$-vertex-critical $(P_5,\text{co-gem})$-free graphs from~\cite{HuangLiXia2023}.
\begin{theorem}\label{thm:cogemP5P3UP2}
There are only finitely many $k$-vertex-critical $(\text{co-gem},P_5, P_3+cP_2)$-free graphs for all $k\ge 1$, and $c \ge 0$. 
\end{theorem}
\begin{proof}
For $k\le 4$, the result follows from Corollary~\ref{cor:finitelymany4critcogem}, so we may assume $k\ge 5$. Further, if $c=0$, then clearly the only $k$-vertex-critical $P_3$-free graph is $K_k$, so we may assume $c\ge 1$.
Therefore, let $k\ge 5$ and $c\ge 1$ be given positive integers and let $G$ be a $k$-vertex-critical $(\text{co-gem},P_5, P_3+cP_2)$-free graph. 
We will prove that $G$ is $(P_3+c'P_1)$-free for 

$$c'=\binom{kc}{\lfloor\frac{kc}{2}\rfloor}.$$ 

Suppose, by way of contradiction, that $G$ contains an induced $P_3+c'P_1$ and let $P=\{p_1,p_2,p_3\}$ and $S=\{s_1,s_2,\ldots,s_{c'}\}$ such that $P\cup S$ induces a $P_3+c'P_1$ where the $P_3$ is in order of the indices.  
Let $M$ be the set of all vertices in $V(G)$ that are mixed on $S$. 
Partition $M$ into sets such that all vertices with the exact same neighbours and non-neighbours in $S$ belong to the same set of the partition. 
Let $U$ be a subset of $M$ defined by taking exactly one vertex from each of the sets in the partition. We now make a series of claims about $U$.

\begin{claim}\label{cla:N(U)containsS}
$S\subseteq N(U)$.
\end{claim}
\begin{proof}[Proof of Claim~\ref{cla:N(U)containsS}]
Since $S$ is an independent set, from Lemma~\ref{lem:nocomparablecliques}, it follows that each vertex in $S$ must have a neighbor which is mixed on $S$. Therefore $S\subseteq N(M)$, and by definition of $U$, it is clear that $N(M)\cap S=N(U)\cap U$. So, $S\subseteq N(U)$.
\end{proof}

\begin{claim}\label{cla:SneighbourinUincomparable}
If $s_i,s_j\in S$ and $i\neq j$, then $N(s_i)\cap U\not\subseteq N(s_j)\cap U$ and $N(s_j)\cap U\not\subseteq N(s_i)\cap U$.
\end{claim}
\begin{proof}[Proof of Claim~\ref{cla:SneighbourinUincomparable}]
Since $S$ is an independent set, it follows from Lemma~\ref{lem:nocomparablecliques} that there is an $m_1\in N(s_i)\setminus N(s_j)$ and an $m_2\in N(s_j)\setminus N(s_i)$, otherwise $s_i$ and $s_j$ would be comparable. Since $m_1$ and $m_2$ are mixed on $S$, it follows that $\{m_1,m_2\}\subseteq M$. If $m_1\not\in U$, then by definition of $U$ there is a $u_1\in U$ such that $N(u_1)\cap S=N(m_1)\cap S$. And similarly, if $u_2\not\in U$, then there is a $u_2\in U$ such that $N(u_2)\cap S=N(m_2)\cap S$. Thus, $N(s_i)\cap U\not\subseteq N(s_j)\cap U$ and $N(s_j)\cap U\not\subseteq N(s_i)\cap U$.
\end{proof}

\begin{claim}\label{cla:Uorderatleastk}
$|U|\ge kc$.
\end{claim}
\begin{proof}[Proof of Claim~\ref{cla:Uorderatleastk}]
By Claims~\ref{cla:N(U)containsS} and~\ref{cla:SneighbourinUincomparable}, it follows that $A=\{N(s)\cap U: s\in S\}$ is an antichain in the partial order $(\mathcal{P}(U),\subseteq)$.

Since $$|A|=c'=\binom{kc}{\lfloor\frac{kc}{2}\rfloor},$$ it follows from Theorem~\ref{thm:Sperner} that $|U|\ge kc$. 
\end{proof}

\begin{claim}\label{cla:nonadjinUhavedistinctSneighbs}
    If $u_1,u_2\in U$ such that $u_1\nsim u_2$, then $N(u_1)\cap S\not\subseteq N(u_2)\cap S$ and $N(u_2)\cap S\not\subseteq N(u_1)\cap S$.
\end{claim}
\begin{proof}[Proof of Claim~\ref{cla:nonadjinUhavedistinctSneighbs}]
 Suppose by way of contradiction that $N(u_1)\cap S\subset N(u_2)\cap S$ or $N(u_2)\cap S\subset N(u_1)\cap S$ (note that we cannot have set equality in either case by the definition of $U$). 
 Without loss of generality, we may assume that that $N(u_1)\cap S\subset N(u_2)\cap S$.
So, there must be vertices $s_i,s_h\in S$ such that $s_i\in \left(N(u_2)\cap S\right)\cap\left(N(u_1)\cap S\right)$ and $s_h\in \left(N(u_2)\cap S\right)\setminus \left(N(u_1)\cap S\right)$.
Since $u_1$ and $u_2$ are both mixed on $S$, it follows that there must also be a vertex $s_j\in S\setminus \left(N(u_1)\cap N(u_2)\right)$.
But now, $\{u_1,s_i,u_2,s_h,s_j\}$ induces a co-gem in $G$, a contradiction.
\end{proof}

\begin{claim}\label{cla:indnumofUatmostc}
$\alpha(G[U]) \le c$. 
\end{claim}
\begin{proof}[Proof of Claim~\ref{cla:indnumofUatmostc}]
    Suppose by way of contradiction that $U$ contains an indpendent set of order at least $c+1$ and let $U'=\{u_1,u_2,\dots u_c,u_{c+1}\}$ be an independent subset of $U$. 
    Suppose there is a vertex $s_h\in S$ such that $u_i\sim s_h$ and $u_j\sim s_h$ for distinct $i,j\in \{1,\dots c+1\}$.  
    By Claim~\ref{cla:nonadjinUhavedistinctSneighbs}, there must also be $s_{i'},s_{j'}\in S$ such that $s_{i'}\sim u_i$, $s_{i'}\nsim u_j$, $s_{j'}\sim u_j$, and $s_{j'}\sim u_i$.
    But now, $\{s_{i'},u_i,s_h,u_j,s_{j'}\}$ induces a $P_5$, a contradiction.

    Therefore, $(N(u_i)\cap S)\cap (N(u_j)\cap S)=\emptyset$ for all distinct $i,j\in \{1,\dots c+1\}$.
    Also, by definition of $U$, every vertex in $U'$ has a neighbour in $S$.
    Without loss of generality (relabelling if necessary), assume $u_i\sim s_i$ for each distinct $i\in \{1,\dots c+1\}$.

    Therefore, $P\cup U' \cup \{s_1,s_2\dots,s_{c+1}\}$ induces a $P_3+cP_2$ unless at least two vertices in $U'$ have neighbours in $P$.
    Without loss of generality (relabelling again if necessary), we may assume $u_1$ and $u_2$ each have neighbours in $P$. 
    
    If $u_1\sim p_1$ or $u_1\sim p_3$, then $u_1\sim p_2$ otherwise $\{u_1,s_1,p_1,p_2,s_2\}$ or  $\{u_1,s_1,p_3,p_2,s_2\}$ induces a co-gem.
    Similarly, if $u_1\sim p_2$, then $u_1\sim p_1$ and $u_1\sim p_3$, otherwise $\{u_1,s_1,p_1,p_2,s_2\}$ or  $\{u_1,s_1,p_3,p_2,s_2\}$ induces a co-gem.
    Thus, $u_1$ is complete to $P$ and by symmetry $u_2$ is also complete to $u_2$.
    
    But now, $\{s_1,u_1,p,u_2,s_2\}$ induces a $P_5$ for any $p\in P$, a contradiction.
\end{proof}

\begin{claim}\label{cla:chromaticnumofUatleastk}
$\chi(G[U])\ge k$. 
\end{claim}
\begin{proof}[Proof of Claim~\ref{cla:chromaticnumofUatleastk}]
    We have $\frac{|U|}{\alpha(G[U])}\ge \frac{kc}{c}=k$ by Claims~\ref{cla:Uorderatleastk} and \ref{cla:chromaticnumofUatleastk}. 
    Further, $\frac{|U|}{\alpha(G[U])}\le \chi(G[U])$ by a folklore bound on the chromatic number. Therefore, we have $\chi(G[U])\ge k$.
\end{proof}

We now complete the proof of the theorem. Claim~\ref{cla:chromaticnumofUatleastk} contradicts $G$ being $k$-vertex-critical since we will have $\chi(G-v)\ge k$ for all $v\in V(G)\setminus U$. So, it must be that $G$ is $(P_3+c'P_1)$-free and therefore there are only finitely many $(\text{co-gem},P_5, P_3+cP_2)$-free graphs for all $k\ge 1$ by Theorem~\ref{thm:finiteP3ellP1freecrit}.
\end{proof}

Since $2P_2$ is an induced subgraph of both $P_5$ and $P_3+cP_2$ for any $c\ge 1$, the following corollary follows immediately form Theorem~\ref{thm:cogemP5P3UP2}.

\begin{corollary}\label{cor:cogem2P2}
    There are only finitely many $k$-vertex-critical $(\text{co-gem},2P_2)$-free graphs for all $k$.
\end{corollary}

\section{$(\text{co-gem, paw}+P_1)$-free graphs}\label{sec:traingleplusP1}

\begin{theorem}\label{thm:cogempawUP1free}
    There are only finitely many $k$-vertex-critical $(\text{co-gem, paw}+P_1)$-free graphs for all $k$. 
\end{theorem}
\begin{proof}
    Let $G$ be a $k$-vertex-critical $(\text{co-gem, paw}+P_1)$-free graph. We will show that $G$ must be $(P_3+2P_1)$-free and it will then follow from Theorem~\ref{thm:finiteP3ellP1freecrit} that there are only finitely many $k$-vertex-critical $(\text{co-gem, paw}+P_1)$-free graphs. Suppose by way of contradiction that $G$ is not $(P_3+2P_1)$-free and let $\{p_1,p_2,p_3,s_1,s_2\}$ induce a $P_3+2P_1$ in $G$ with $p_1p_2p_3$ being an induced $P_3$, in that order. From Lemma~\ref{lem:nocomparablecliques}, we must have a vertex $s_1'\in V(G)$ such that $s_1'\sim s_1$ but $s_1'\nsim s_2$. We now consider four cases depending on the neighbours of $s_1'$ in $\{p_1,p_2,p_3\}$.\\   

    \noindent\textit{Case 1.} $s_1'\sim p_2$ and at least one of $s_1'\sim p_1$ or $s_1'\sim p_3$.\\

    \noindent Suppose without loss of generality that $s_1'\sim p_1$. Then $\{p_2,p_1,s_1',s_1,s_2\}$ induces a $\text{paw}+P_1$, a contradiction. \\ 

    \noindent\textit{Case 2.} $s_1'\sim p_2$, $s_1'\nsim p_1$, and $s_1'\nsim p_3$.\\

    \noindent Then $\{s_1,s_1',p_2,p_1,s_2\}$ induces a co-gem, a contradiction. \\ 
    
    \noindent\textit{Case 3.} $s_1'\nsim p_2$ and at least one of $s_1'\sim p_1$ or $s_1'\sim p_3$.\\

    \noindent Suppose without loss of generality that $s_1'\sim p_1$. Then $\{p_2,p_1,s_1',s_1,s_2\}$ induces a co-gem, a contradiction. \\

     \noindent\textit{Case 4.} $s_1'\nsim p_2$, $s_1'\nsim p_1$, and $s_1'\nsim p_3$.\\

    \noindent Again from Lemma~\ref{lem:nocomparablecliques}, we must have $p_1'\in V(G)$ such that $p_1'\sim p_1$ but $p_1'\nsim p_3$.   Suppose that $p_1'\nsim x$ for some $x\in \{s_1,s_2,s_1'\}$. Now, if $p_2\sim p_1'$, then $\{p_2,p_1,p_1',p_3,x\}$ induces a $\text{paw}+P_1$, and if $p_2\nsim p_1'$, then $\{p_3,p_2,p_1,p_1',x\}$ induces a co-gem. Thus, $p_1'$ is complete to $\{s_1,s_2,s_1'\}$. 
    But now  $\{p_1',s_1,s_1',s_2,p_3\}$ induces a $\text{paw}+P_1$ in $G$, a contradiction.\\

    Since we reach a contradiction in each of the four cases above, we contradict the fact that $s_1'$ exists. Since $G$ is $k$-vertex-critical, we have already noted that $s_1'$ must exist. Therefore, this contradicts $G$ having an induced $P_3+2P_1$ and completes the proof. 
\end{proof}

Since $K_3+P_1$ is an induced subgraph of $\text{paw}+P_1$ the following corollary is immediate.

\begin{corollary}\label{cor:cogemtriangleP1}
There are only finitely many $k$-vertex-critical $(\text{co-gem, }K_3+P_1)$-free graphs for all $k$. \hfill $\square$
\end{corollary}

\section{(co-gem, $K_4$)-free graphs}

Our result in this section relies completely on exhaustive computer search using Goedgebeur and Schaudt's program CriticalPfreeGraphs available at~\cite{CriticalPfreeGraph}. In short (and with oversimplification), the algorithm starts with a set of induced subgraphs that must be present in a vertex-critical graph and then extends the set in all possible ways, removing any graphs that have induced subgraphs that are forbidden by the user or when the chromatic number reaches $k$. The $k$-vertex-critical graphs are returned as a subset of this set (see Algorithms 1 and 2 in~\cite{GoedgebeurSchaudt2018} for full details).  The initial ideas for the algorithm were first presented in~\cite{Hoang2015} and used to generate all $5$-vertex-critical $(P_5,\ C_5)$-free graphs. This algorithm was then used as a starting point for an improved version that was used in~\cite{4critconnectedSODA2016} (see also the full version of the extended abstract~\cite{Chudnovsky4criticalconnected2020}) to show that there are only finitely many $4$-vertex-critical $P_6$-free graphs and give the complete list of them all. The algorithm was then thoroughly extended and optimized in~\cite{GoedgebeurSchaudt2018} where its strength was demonstrated by showing many new results on the finiteness of $4$-vertex-critical graphs in various families, including  $(P_7,\ C_5)$-free, $(P_8,\ C_4)$-free, $(P_4+2P_1)$-free, and $(P_3+P_2)$-free. Among the results from the new algorithm in~\cite{GoedgebeurSchaudt2018} were results on the colourability of various classes of graphs proved completely by the algorithm returning an empty list of $k$-vertex-critical graphs for various $k$. Some of these results included that the following graphs are $3$-colourable: $P_{11}$-free graphs of girth at least five, $P_{14}$-free graphs of girth at least six, $P_{17}$-free graphs of girth at least seven. Since the extension and optimization in~\cite{GoedgebeurSchaudt2018}, the algorithm has been used to prove many results on vertex-critical graphs, some of the most relevant to our discussions are:

\begin{itemize}
    \item Aiding with the proof that there are only finitely many $k$-vertex-critical $(P_5,\ K_4)$-free graphs for all $k$~\cite{KCameron2021},
    \item There is only one $6$-vertex-critical $(P_6\text{, diamond})$-free graph with clique number $3$, which was used in the proof that the chromatic number of $(P_6\text{, diamond, }K_4)$-free graphs can be computed in polynomial-time~\cite{GoedgebeurHuangJuMerkel2023}, and
    \item Classifying all $k$-vertex-critical $(P_5\text{, dart})$-free graphs for $k\in\{4,5,6,7\}$~\cite{Xiaetal2023}.
\end{itemize}

\begin{theorem}\label{thm:nocogemK45crit}
    There are no $5$-vertex-critical $(\text{co-gem}, K_4)$-free graphs.
\end{theorem}
\begin{proof}
    The program CriticalPfreeGraphs terminates with an empty list of $5$-vertex-critical $(\text{co-gem}, K_4)$-free graphs. More specifically, after $35$\footnote{Thus, the 64-bit version of the program is required as the standard version only support graphs of maximum order $32$.} vertices, the algorithm terminates with no possible extensions and no $5$-vertex-critical graphs generated.
\end{proof}

Since every graph with chromatic number at least $5$ must contain an induced $5$-vertex-critical subgraph and being $(\text{co-gem}, K_4)$-free is a hereditary property, we get the following corollary immediately.

\begin{corollary}\label{cor:cogemK4free4colourable}
    Every $(\text{co-gem}, K_4)$-free graph is $4$-colourable.
\end{corollary}

\section{Proof of Theorem~\ref{thm:cogemHfreeord4}}\label{sec:proofofmainthm}
We are now ready to prove our main theorem.

\begin{proof}[Proof of Theorem~\ref{thm:cogemHfreeord4}]
For the graphs names of order four in this proof, please refer to Figure~\ref{fig:graphsoforder4}. Let $H$ be a graph of order $4$. For $k=1,2$, the result is trivial. For $k=3$, the only $k$-vertex-critical co-gem-free graphs are $K_3$ and $C_5$, since every other odd cycle contains induced co-gems. For $k=4$, the result is known from Corollary~\ref{cor:finitelymany4critcogem}. 

To complete the proof, for all $k\ge 5$, there are only finitely many $k$-vertex-critical $(\text{co-gem, } H)$-free graphs when $H$ is:

\begin{itemize}
    \item $\overline{K_4}$ by Ramsey's Theorem~\cite{Ramsey};
    \item $P_2+2P_1$ since it was shown in~\cite{CameronHoangSawada2022} that there are only finitely many $k$-vertex-critical $(P_2+2P_1)$-free graphs for all $k$,
    \item $P_3+P_1$ since it was shown in~\cite{KCameron2021} that there are only finitely many $k$-vertex-critical $(P_3+P_1)$-free graphs for all $k$ (see also the proofs in \cite{CameronHoangSawada2022} and \cite{AbuadasCameronHoangSawada2022});
    \item $2P_2$ from Corollary~\ref{cor:cogem2P2};
    \item claw since it was shown in~\cite{Kaminski2019} that there are only finitely many $k$-vertex-critical $(P_6, \text{claw})$-free graphs for all $k$; 
    \item $P_4$ since $P_4$-free graphs are perfect as outlined in the introduction;
    \item $K_3+P_1$ from Corollary~\ref{cor:cogemtriangleP1};
    \item paw by Theorem~\ref{thm:gemcogem} since paw is an induced subgraph of gem (note that it also now follows from Theorem~\ref{thm:cogempawUP1free} with a much shorter proof);
    \item $C_4$ since it was shown in~\cite{Kaminski2019} that there are only finitely many $k$-vertex-critical $(P_6, C_4)$-free graphs for all $k$;
    \item diamond by Theorem~\ref{thm:gemcogem} since diamond is an induced subgraph of gem;
    \item $K_4$ by Corollary~\ref{cor:cogemK4free4colourable}.
    
\end{itemize}

\end{proof}

\section{Conclusion}\label{sec:conclusion}

While Theorem~\ref{thm:cogemHfreeord4} is a good step forward for the study of co-gem free graphs, the finiteness of $k$-vertex-critical co-gem-free graphs remains an open problem for all $k\ge 5$. Again using the CritcalPfreeGraphs~\cite{CriticalPfreeGraph} program we were able to generate all $5$-vertex-critical co-gem-free graphs of order at most 18 and we find that there are 327 total graphs available in graph6 format at~\cite{cogemfreefreefiles}; 1 of order 5, 1 of order 7, 7 of order 8, 228 of order 9, 70 of order 10, 16 of order 11, and 4 of order 12. Given that there are no $5$-vertex-critical co-gem-free graphs of order $n$ for $13\le n\le 18$, we pose the following conjecture.

\begin{conjecture}
    There are only finitely many $5$-vertex-critical co-gem-free graphs, and the list of 327 at~\cite{cogemfreefreefiles} is complete.
\end{conjecture}

The natural next step after Theorem~\ref{thm:cogemHfreeord4} would be to solve the following open problem.

\begin{problem}
    For which graphs $H$ of order $5$ are there only finitely many $k$-vertex-critical $(\text{co-gem, } H)$-free graphs for all $k$?
\end{problem}

One such $H$ of order $5$ that might be a good starting place is bull, since using the CriticalPfreeGraphs program again we find that all $k$-vertex-critical $(\text{co-gem, bull})$-free graphs are exactly the same as the $k$-vertex-critical $(P_3+P_1)$-free graphs for all $k\le 6$.

In light of Theorem~\ref{thm:nocogemK45crit}, it may seem reasonable to conjecture that there are no $6$-vertex-critical $(\text{co-gem, }K_5)$-free graphs, but that turns out to be false. Using the same program to generate such graphs, we find that there are 1479 such graphs of order at most $12$ (1 of order $10$, $111$ of order $11$, and $1367$ of order $12$). However, there are no $6$-vertex-critical $(\text{co-gem, } C_5,\ K_5)$-free graphs, thus we get the following:

\begin{proposition}
    Every $(\text{co-gem, }C_5,\ K_5)$-free graphs is $5$-colourable.
\end{proposition}

In each case, the set of forbidden induced subgraphs leave only one odd antihole to build all of the imperfect $k$-vertex-critical graphs on from the Strong Perfect Graph Theorem~\cite{Chudnovsky2006}. This leads us to the following conjecture.

\begin{conjecture}\label{conj:finalconj}
    Every $(\text{co-gem, } \overline{C_{5}},\overline{C_{7}},...,\overline{C_{2k-5}}, K_k)$-free graph is $k$-colourable.
\end{conjecture}

\section*{Acknowledgements}
The second author gratefully acknowledges research support from the Natural Sciences and Engineering Research Council of Canada (NSERC), grants RGPIN-2022-03697 and DGECR-2022-00446. Both authors also thank Jan Goedgebeur for helpful discussions on using their program and for even going so far as to run some computations for us. We also acknowledge Angeliya C. U., S. A. Choudum, and Mayamma Joseph for finding a counterexample in a conjecture in an earlier version of this paper.

\bibliographystyle{abbrv}
\bibliography{refs}

\end{document}